\documentclass[11pt]{article}
\usepackage{latexsym,amsmath,amscd,amssymb,graphics,mathrsfs}
\usepackage{enumerate}
\usepackage{graphicx}
\usepackage{color}
\usepackage{url}

\newcommand{\rem}[1]{}
\makeatletter

\newcommand\be{\beta}
\newcommand\ga{\gamma}
\newcommand\de{\delta}

\newcommand\ze{\zeta}

\newcommand\la{\lambda}
\newcommand\rh{\rho}

\newcommand\si{\sigma}
\newcommand\ta{\tau}
\newcommand\ph{\varphi}

\newcommand\om{\omega}
\newcommand\Ga{\Gamma}

\newcommand\Ph{\Phi}
\newcommand\Ps{\Psi}
\newcommand\Om{\Omega}

\newcommand\ie{i.e.\ }

\newcommand\oo{{\infty}}

\renewcommand\o{\circ}
\renewcommand\div{\on{div}}
\newcommand\x{\times}
\newcommand\on{\operatorname}
\newcommand\Ad{\on{Ad}}

\newcommand\Conf{\on{Conf}}
\newcommand\Emb{\on{Emb}}

\newcommand\can{\on{can}}

\newcommand\Rot{\on{Rot}}

\newcommand\Diff{\on{Diff}}

\newcommand\ev{\on{ev}}

\newcommand\pr{\on{pr}}

\newcommand\Ham{\on{Ham}}
\newcommand\KKS{\on{KKS}}
\newcommand\curl{\on{curl}}
\renewcommand\prod{\on{prod}}

\newcommand\g{\mathfrak g}

\newcommand\h{\mathfrak h}

\newcommand\ZZ{\mathbb Z}
\newcommand\NN{\mathbb N}

\newcommand\RR{\mathbb R}

\newcommand\X{\mathfrak X}

\renewcommand\S{\mathcal S}
\renewcommand\O{\mathcal O}

\newcommand\w{\textsf{w}}

\newenvironment{proof}[1][Proof]{\noindent\textbf{#1.} }{\ \rule{0.5em}{0.5em}}

\def\XXint#1#2#3{{\setbox0=\hbox{$#1{#2#3}{\int}$ }
\vcenter{\hbox{$#2#3$ }}\kern-.5\wd0}}

\begin{document}

\newtheorem{theorem}{Theorem}[section]
\newtheorem{definition}[theorem]{Definition}
\newtheorem{lemma}[theorem]{Lemma}
\newtheorem{remark}[theorem]{Remark}
\newtheorem{proposition}[theorem]{Proposition}
\newtheorem{corollary}[theorem]{Corollary}
\newtheorem{example}[theorem]{Example}

\definecolor{burgund}{RGB}{153,0,51}      %burgundy

%%%%%%%%%%%%%%%%%%%%%%%%%%%%%%%%%%%%%%%%%%%%%%
%%%%%%%%%%%%%%%%%%%%%%%%%%%%%%%%%
%%%%%%%%

\title{Pointed vortex loops in ideal 2D fluids}
\author{Ioana Ciuclea$^{1}$ and Cornelia Vizman$^{1}$ }
\addtocounter{footnote}{1}

\footnotetext{Department of Mathematics,
West University of Timi\c soara, 
300223-Timi\c soara, Romania.
\texttt{cornelia.vizman@e-uvt.ro}
\addtocounter{footnote}{1} }

\date{ }
\maketitle
\makeatother
%\begin{center} DRAFT \end{center}
%\maketitle

%|||-------------------text width----------------------|||

%\noindent \textbf{AMS Classification:} 53D20; 37K65; 58D10

%\noindent \textbf{Keywords:} \textcolor{black}{coadjoint orbits, nonlinear Grassmannians, prequantization, characters.}

\begin{abstract}
We study a special kind of singular vorticities in ideal 2D fluids that combine features of point vortices and vortex sheets, namely pointed vortex loops.
We focus on the coadjoint orbits of the area-preserving diffeomorphism group of $\RR^2$ determined by them.
We show that a polarization subgroup consists of diffeomorphisms that preserve the loop as a set,
thus the configuration space is the space of loops that enclose a fixed area, without information on vorticity distribution and  attached points.
\end{abstract}
%\tableofcontents

\section{Introduction}

Euler's equations $\dot {\bf v}+\nabla_{\bf v}{\bf v}=-\nabla p$, $\div {\bf v}=0$, for ${\bf v}$ the fluid velocity and $p$ the pressure,
are  geodesic equations on the volume preserving diffeomorphism group \cite{Arnold,EM}.
In particular the vorticity $\curl {\bf v}$ is confined to a coadjoint orbit of this group.

Among singular vorticities for ideal fluids in 2D, the point vortices
have been much studied (see \cite{aref} for a survey).
In \cite{MW} they are set in the framework of momentum maps:
point vortices evolve in coadjoint orbits of the area preserving diffeomorphism group.
These consist of $k$-tuples of distinct points 
%$x_1,\dots,x_k$ 
in the plane, and the Kostant-Kirillov-Souriau (KKS) symplectic form 
is a linear combination of the area forms weighted with the vorticities (circulations) $\Ga_1,\dots,\Ga_k$.
If some of the $\Ga_i$'s coincide, then the quotient by a permutation subgroup of $k$ elements has to be taken
(Proposition \ref{permit}).
% in the plane (because of invariance of their vorticities). $\{x_1,\dots,x_N\}$ 
%The symmetry group here is the group of compactly supported area-preserving diffeomorphisms of the plane
%(maybe mention Arnold here an put it as the first sentence).

Another kind of singular vorticity is the vortex sheet \cite{batchelor},
e.g. the vortex loop: a pair $(C,\be)$ with $C$ a closed curve and $\be$ its vorticity density (strength), assumed to be nowhere zero \cite{goldin}.
The corresponding coadjoint orbit $\O_a^{\w}$ of the area preserving diffeomorphism group 
consists of curves that enclose a fixed area $a$, endowed with vorticity densities of fixed total vorticity $\w=\int_C\be$ (Theorem \ref{below}).
One can view these spaces as isodrastic leaves in the space of weighted Lagrangian submanifolds, showed in \cite{W90}
to be coadjoint orbits of the Hamiltonian group.
Another way to get the same coadjoint orbits is by using results from \cite{GBV} on coadjoint orbits of the volume preserving diffeomorphism group consisting of vortex sheets,
results that provide also the prequantization condition $a\w\in 2\pi\ZZ$ of the KKS symplectic form (see also \cite{GBVpre}).
Here we obtain the same coadjoint orbit $\O_a^{\w}$ by symplectic reduction on one of the legs of the ideal fluid dual pair
due to Marsden and Weinstein \cite{MW}.
%This leg is the momentum map given by the enclosed area.
 
In this article we study a different type of singular vorticities, that combine features of point vortices and vortex loops, called
pointed vortex loops.
A complete set of invariants for the natural action of the area preserving diffeomorphism group
on pointed vortex loops are the enclosed area $a$ and the $k$ partial vorticities
$\w_i=\int_{x_i}^{x_{i+1}}\be$ along the loop, with total vorticity $\w_1+\dots \w_k=\w$.
The corresponding coadjoint orbit can be identified with the space of embeddings 
that enclose a fixed area, possibly factorized by a permutation subgroup of $k$ elements, depending on the rotational symmetry of the data, 
$(\Ga_i)$ and $(\w_i)$ (Theorem \ref{main}).
The KKS symplectic form is exact.
%, which depend on the point vorticities $\Ga_i$ and on the partial vorticities $\w_i$, are all exact.
At the pointed vortex loop $(C,\be,(x_i))$ it can be written in canonical form, with a well chosen non-degenerate pairing between
$C^\oo(C)$ and its subspace $C_0^\oo(C)$ of zero integral functions (with respect to the volume form on $C$ induced by the Euclidean metric).
%We notice that the coadjoint orbit of pointed vortex loops can be embedded as a symplectic submanifold in the product of other two coadjoint orbits: of vortex loops and of point vortices.

%\paragraph{Polarizations.}
In \cite{goldin} is shown that, within the framework of the 2D Euler equations, point vortices cannot be consistently quantized, 
while vortex loops admit natural polarizations. 
The resulting configuration space is the space of loops enclosing a fixed area, without information about the vorticity distribution.
In the same article the authors attach vortex dipoles to the point vortices, 
so that the additional degrees of freedom allow for a polarization.
We show that pointed vortex loops also have natural polarizations.
A polarization subgroup consists of diffeomorphisms that preserve the loop as a set,
so the configuration space is the space of loops that enclose a fixed area,
without information about the vorticity distribution and the attached points.
%It is the same as for coadjoint orbits of vortex loops without added points.
%Thus we don't have a quantum theory of pointed vortex loops 
%from quantizing singular point vortices on the loop.
%For this one has to quantize  point vortex dipoles on the loop, as in \cite{GBGVpre}.

\paragraph{Acknowledgements.} 
We would like to thank Gerald Goldin and Francois Gay-Balmaz for interesting (thought-provoking) discussions on the subject and useful suggestions.
We also thank Yael Karshon for her question that became Remark \ref{yael}.
Both authors were supported by a grant of the Romanian Ministry of
Education and Research, CNCS-UEFISCDI, project number PN-III-P4-ID-PCE2020-2888, within PNCDI III. 
%%%%%%%%%%%%%%%%%%%%%%%%

\section{Singular vorticities}

In this section we describe coadjoint orbits for two types of singular vorticities for ideal fluids in 2D: point vortices and vortex loops.

We denote by $\om$ the canonical area form on $\RR^2$ (a symplectic form) and by $G$ the group of compactly supported area preserving diffeomorphisms, which coincides with $\Ham_c(\RR^2)$, the compactly supported Hamiltonian diffeomorphism group.
Its Lie algebra, denoted by $\g$, consists of compactly supported divergence free vector fields
(the same as compactly supported Hamiltonian vector fields in this case). 
It can be identified with the Lie algebra $C_c^\oo(\RR^2)$ and its 
%(endowed with the Poisson bracket), hence its
 dual $\g^*$ with the space of distributions.
% (generalized functions).

Euler's equations in the plane, $\dot {\bf v}+\nabla_{\bf v}{\bf v}=-\nabla p$, $\div {\bf v}=0$, for ${\bf v}$ the fluid velocity and $p$ the pressure,
are the geodesic equations on the group $G$ endowed with the right invariant $L^2$ metric \cite{Arnold}.
Here ${\bf v}$ lives in $\g$, while the vorticity $\curl {\bf v}$ 
is confined to a coadjoint orbit in $\g^*$.
Beside the smooth (regular) vorticities, non-smooth (singular) vorticities were widely considered.
These are confined to coadjoint orbits too  (see for instance \cite{K}).
In 2D one has point vortices, with zero dimensional support, and vortex loops, with 1-dimensional support (\ie vortex sheets).

\paragraph{Point vortices.}
The configuration space of $k$ ordered points in the plane is
\begin{equation}\label{conf}
{\Conf_k}= (\RR^2)^k\setminus\Delta_k,
\end{equation}
where $\Delta_k = \{(x_1,...,x_k) \in (\mathbb R^2)^k : x_i=x_j \text{ for some } i\neq j\}$ is the fat diagonal.
The non-zero vorticities (circulations) $\Ga_i\in\RR$ for $i=1,\dots k$, define a symplectic form on $\Conf_k$  
\begin{equation}\label{omsum}
\om^\Ga=\sum_{i=1}^k\Ga_i p_i^*\om,\quad p_i:\Conf_k\to\RR^2.
\end{equation}
The natural action of  $\Ham_c(\RR^2)$  
is transitive and Hamiltonian with equivariant momentum map \cite{MW}
\begin{equation}\label{jpoint}
J:\Conf_k\to C_c^\oo(\RR^2)^*,\quad J(x_1,\dots,x_k)=\sum_{i=1}^k\Ga_i\de_{x_i}.
\end{equation}
In the generic case, when all vorticities are distinct, the map $J$ is one-to-one onto a coadjoint orbit and the KKS symplectic form on the coadjoint orbit satisfies 
$J^*\om_{KKS} = \om^\Ga$.
If some of the vorticities coincide, let's say $\Ga_{i_1}=\dots =\Ga_{i_{k'}}$, then the permutation group $\S_{k'}$ of the set of indices $\{i_1,\dots,i_{k'}\}$ acts on the configuration space while preserving the symplectic form $\om^\Ga$, 
and we have to quotient out this action to ensure the injectivity of $J$.
In the most symmetric case, when all vorticities are equal, factorization by the whole permutation group $\S_k$ is needed,
and the coadjoint orbit is the configuration space of $k$ unordered points in the plane.
The general case is described in the following proposition:
% (mathematical folklore?):

\begin{proposition}\label{permit}
Let $\Ga_1,\dots,\Ga_k$ be fixed vorticities,
let $K$ be the partition of $\{1,\dots,k\}$ that corresponds to equal vorticities,
%: let's say $m$ parts of cardinalities $k_1,\dots,k_m$.
and let 
\[
\S_K:=\S_{k_1}\x\dots\x\S_{k_m}
\] 
be the subgroup of $\S_k$ that consists of all permutations that preserve the partition $K$.
Then the action of $\S_K$ on the configuration space $\Conf_k$ preserves the symplectic form $\om^\Ga$ in \eqref{omsum}
and the quotient space $\Conf_k/\S_K$ inherits a symplectic form ${\bar\om}^\Ga$.
The momentum map for the Hamiltonian action of $\Ham_c(\RR^2)$,
\[
\bar J:\Conf_k/\S_K\to C^\oo_c(\RR^2)^*,\quad\bar J([x_1,\dots,x_k])=\sum_{i=1}^k\Ga_i\de_{x_i},
\]
is one-to-one onto a coadjoint orbit and the KKS symplectic form on the coadjoint orbit satisfies 
${\bar J}^*\om_{KKS} = {\bar\om}^\Ga$.
\end{proposition}

\paragraph{Vortex loops.}
We call a vortex loop any weighted curve $(C,\be)$, with $C$ a closed curve in the plane endowed with a nowhere zero vorticity density $\be\in\Om^1(C)$,
assumed to be a volume form.
We consider the orientation of the curve induced by $\be$, so the total vorticity $\w=\int_C\be$ is positive.
This a special type of vortex sheet in two dimensions \cite{batchelor}. A formulation within the groupoid framework can be found in \cite{ik}.

Each such pair $(C,\be)$ corresponds to a unique non-smooth element of the dual $\g^*$, 
\begin{equation}\label{cbe}
\langle(C,\be),X_h\rangle=\int_Ch\be, \quad h\in C^\oo_c(\RR^2),
\end{equation}
where the compactly supported Hamiltonian vector field $X_h\in\g$ is identified with its unique compactly supported hamiltonian function $h$.
The enclosed area and the total vorticity are a complete set of invariants of the coadjoint action
$\Ad^*_\ph(C,\be)=(\ph(C),\ph_*\be)$ for all $\ph\in G$ (see Theorem \ref{below} below).

%\paragraph{Dual pair.}
The ideal fluid dual pair, due to Marsden and Weinstein \cite{MW} (and further studied in \cite{GBV12}), has a low dimensional version 
that captures vortex loops, which we present below. The manifold $\Emb(S^1,\RR^2)$ 
carries a  symplectic form naturally defined with the volume form $\mu=\frac{\w}{2\pi}dt$
 on $S^1$ and the symplectic form $\om=d\nu$ on $\RR^2$:
\begin{equation}\label{omf}
\Om_f(u_f,v_f)=\int_{S^1}\om(u_f,v_f)\mu,\quad u_f,v_f:S^1\to\RR^2.
\end{equation}
The commuting Hamiltonian actions of $\Ham_c(\RR^2)$ from the left and $\Rot(S^1)$ from the right
admit equivariant momentum maps:
$J_L(f)=f_*\mu\text{ and }J_R(f)=\int_{S^1}f^*\nu$ the area enclosed by the image of $f$.
%The right momentum map is the area, while the left momentum map applied to $f$ is the distribution defined by $(C,\be)=(f(S^1),f_*\mu)$ as in \eqref{cbe}.
They form a symplectic dual pair
\[
\g^*\stackrel{J_L}{\longleftarrow}(\Emb(S^1,\RR^2),\Om)\stackrel{J_R}{\longrightarrow}\RR,
\] 
which means that the distributions $\ker T J_L$ and $\ker T J_R$ 
are symplectic orthogonal complements of one another \cite{W83}.
A slightly stronger fact holds here: each of the two groups acts transitively on level sets of the other group's momentum map.
Thus the orbits are symplectic complements of one another (this type of actions are called mutually completely orthogonal in \cite{LM}).

For such dual pairs of momentum maps, a general principle says that symplectic reduction on one leg yields coadjoint orbits for the other group.
Symplectic reduction at zero in the ideal fluid dual pair has been already used in \cite{GBV19} to obtain
coadjoint orbits of the Hamiltonian group, orbits that consist of weighted isotropic submanifolds.
%(see also \cite{L}).
The symplectically reduced manifold for the right action at  non-zero $a\in\RR$, which is
\begin{equation}\label{cong}
(\Emb_a(S^1,\RR^2)/\Rot(S^1),\Om_a^{\w})
\end{equation}
can be realized via $J_L$ as a coadjoint orbit of $G$.
It is in one-to-one correspondence  with the space
$\O_a^{\w}$ of all vortex loops $(C,\be)$ with total vorticity  $\int_C\be=\w$ and the enclosed area $\int_C\nu=a$. 
More precisely,
\begin{equation}\label{psi}
\Ps:\Emb_a(S^1,\RR^2)/\Rot(S^1)\to\O_a^{\w},\quad \Ps([f])=(f(S^1),f_*\mu)
\end{equation}
is a bijection that intertwines the natural $G$ action $\ph\cdot[f]=[\ph\o f]$ with the $\Ad^*$ action.

%\paragraph{Isodrasts.}
The same result can be deduced by adapting to the case of curves in the plane the results from 
%In the language of Weinstein 
\cite{W90} (see also \cite{L}), where isodrastic leaves in the space of weighted Lagrangian submanifolds are shown to be
coadjoint orbits of the Hamiltonian group,
The name isodrast refers to the same action: 
the action integral around loops in Lagrangian submanifolds is preserved under isodrastic deformations.
In this 2D setting, 
%the isodrasts of the space of Lagrangian embeddings $\Emb(S^1,\RR^2)$ are  $\Emb_a(S^1,\RR^2)$, $a\in\RR$, hence 
the isodrasts of the space of weighted 1-dimensional Lagrangian submanifolds of $\RR^2$ with total weight $\w$ are the coadjoint orbits $\O_a^{\w}$
with $a\in\RR$.

A third way to arrive to these coadjoint orbits is by adapting to dimension two the results from \cite{GBV} on vortex sheets, \ie
singular vorticities of codimension one (see also \cite{GBVpre}).
One also gets which of the coadjoint orbits $(\O_a^{\w},\Om_a^{\w})$ in \eqref{cong} are prequantizable:
those that satisfy  the Onsager-Feynman prequantization condition $a\w\in 2\pi\ZZ$ (see also \cite{goldin}).
%We get a linear family of symplectic structures $\Om_a^{\w}$ on the same manifold $\Emb_a(S^1,\RR^2)/\Rot(S^1)$  have as KKS forms integer multiples of this symplectic form.

We summarize these facts in the following theorem:

\begin{theorem}\label{below}{\rm\cite{W90,L,GBV}}
The space $\O_a^{\w}$ of vortex loops with enclosed area $a$ and total weight $\w$, identified with $\Emb_a(S^1,\RR^2)/\Rot(S^1)$
% via $\Ps$ in \eqref{psi}, 
and endowed with reduced symplectic form $\Om_a^{\w}$, has a natural Hamiltonian $G$ action.
The momentum map
\begin{equation}\label{jloop}
J:\Emb_a(S^1,\RR^2)/\Rot(S^1)\to\g^*, \quad\langle J([f]),X_h\rangle=\frac{\w}{2\pi}\int_{S^1}h(f(t))dt
\end{equation}
is one-to-one onto a coadjoint orbit and the KKS symplectic form on the coadjoint orbit satisfies 
$J^*\om_{KKS} = \Om_a^{\w}$. The coadjoint orbit  $\O_a^{\w}$ is prequantizable if and only if $aw\in 2\pi\ZZ$.
\end{theorem}

%%%%%%%%%%%%%%%

\section{Pointed vortex loops}

We call pointed vortex loop  a triple $(C,\beta,(x_i))$ that consists of a vortex loop  $(C,\be)$ to which we attach $k$ point vortices
 $x_i\in C$.
 In this section we study singular vorticities in $\g^* = C_c^\infty(\mathbb R^2)^*$ induced by pointed vortex loops:
\begin{equation}\label{cbx}
X_h\mapsto\int_Ch\beta + \sum_{i=1}^k\Gamma_ih(x_i).
\end{equation}

\paragraph{The invariants of the $G$ action.}
Beside the total vorticity $\w = \int_C\beta>0$, the point vortices introduce $k$ additional partial vorticities along the loop:
$$
\w_i = \int_{x_i}^{x_i+1}\beta \quad \text{ with } \quad \w = \w_1 + ... + \w_k.
$$
We only consider consecutive points on the curve, while respecting the orientation induced by $\be$, which means that all $\w_i>0$.
The natural action of $G=\Ham_c(\RR^2)$,
% given by the pushforward,
\begin{equation}\label{adcx}
\varphi\cdot((C,\beta, ( x_i))=(\varphi(C), \varphi_*\beta,(\varphi( x_i)),
\end{equation}
leaves invariant the area $a$ enclosed by the curve, as well as all partial vorticities $\w_i $.
We denote by $\O_a^{\bar\w}$ the space of all pointed vortex loops $(C,\be,(x_i))$ with 
$a=\int_C\nu$ and $\w_i = \int_{x_i}^{x_i+1}\beta$, $i=1,\dots,k$.

\begin{lemma}\label{pfi}
Consider the manifold $\Emb_a(S^1,\RR^2)$  of embeddings 
that enclose a fixed area $a$, and fix a collection of positive numbers $\bar\w=(\w_i)$ with $\w=\w_1+\dots+\w_k$. 
Then the map
\begin{equation}\label{fie}
\Phi: \Emb_a(S^1, \RR^2) \rightarrow \mathcal O_a^{\bar\w}, \quad\Phi(f) = (f(S^1), f_*\mu, (f( t_i)))
\end{equation}
is a bijection, 
where $\mu:=\frac{\w}{2\pi}dt$ and, for $i=1,\dots,k$,
\begin{equation}\label{ti}
t_i:=\frac{\w_1+\dots+\w_{i-1}}{\w}2\pi.
\end{equation}
\end{lemma}

\begin{proof}
The map $\Phi$ is well defined since 
\[
\int_{f(S^1)}f_*\mu=\frac{\w}{2\pi}\int_{0}^{2\pi}dt=\w,\quad\int_{f(t_i)}^{f(t_{i+1})}f_*\mu=\frac{\w}{2\pi}\int_{t_i}^{t_{i+1}}dt=\w_i.
\]
It is injective because $\Phi(f_1) = \Phi(f_2)$ implies that the two parametrizations differ by a rigid rotation:
$f_2 = f_1\circ R_{\tau}$. From $f_1(0) = f_2(0)$ (since $t_1=0$) we infer that the rotation $R_\ta$ must be the identity map, so $f_1=f_2$.

For the surjectivity we consider an arbitrary pointed vortex loop  $(C,\beta,(x_i)) \in \O_a^{\bar\w}$.
There is a parametrization $f\in \Emb_a(S^1, \RR^2)$ such that $C = f(S^1)$ and $\beta = f_*\mu$.
We use the freedom to compose the embedding $f$ with a rigid rotation to get $x_1=f(0) = f(t_1)$. It remains to be shown that $x_i=f(t_i)$ for $i=2,\dots,k$. This follows from
\[
\int_{x_1}^{x_i}\be=\w_1+\dots \w_{i-1}=\frac{\w}{2\pi}t_i=\int_0^{t_i}\mu
%=\int_{f(0)}^{f(t_i)}f_*\mu
=\int_{x_1}^{f(t_i)}\be.
\]
Thus $\Ph(f)=(C,\be,(x_i))$.
\end{proof}

The bijection $\Ph$ in \eqref{fie} intertwines the $G$ action $\ph\cdot f=\ph\o f$ with the natural action \eqref{adcx}.
It is well known  that the  action of $\Ham_c(\RR^2)$ on $\Emb_a(S^1, \RR^2)$ is transitive (see \cite{GBV} or Proposition
\ref{transitivity} from the Appendix). 
With the Lemma \ref{pfi} we obtain that: 

\begin{proposition}
The space of pointed vortex loops $\O_a^{\bar\w}$ is a $\Ham_c(\RR^2)$ orbit,
thus the enclosed area $a$ and the partial vorticities $\bar\w$ are a complete set of invariants.
\end{proposition}

%\todo{1}
\begin{remark}\rm
Another expression of the tangent space to the orbit of pointed vortex loops $\O_a^{\bar\w}$, 
in terms of triples (inspired by the vortex sheet approach in \cite{GBV})
uses the orthogonal decomposition of vector fields along the curve (with respect to the ambient Euclidean metric).
The infinitesimal generator to $X_h \in \g$, which is the vector field $f\mapsto X_h\o f$ on the space of embeddings,
becomes
\[
\ze_{X_h}(C,\be,(x_i))=(X_h|_C^\perp,di_{X_h|_C^\top}\be,(X_h(x_i))).
\]
Here the tangent space at $(C,\be,(x_i))$ to its $G$ orbit $\O_a^{\bar\w}$ is identified with
\begin{align}\label{trip}
\{(u_C,d\la,(v_i))\in\Ga(TC^\perp)\x dC^\oo(C)\x\RR^{2k}: i_{u_C}\om|_{TC}\in dC^\oo(C),&\ v_i^\perp=u_C(x_i),\nonumber\\
&\be(v_i^\top)-\la(x_i)={\rm ct.}\}.
\end{align}
% $(u_C,d\la,(v_i))$ at $(C,\be,(x_i))$ is subject to additional constraints:
The first condition comes from the constant enclosed area.
The second condition doesn't let the points $x_i$ to leave the curve.
The requirement that the local vorticity $\w_i=\int_{x_i}^{x_{i+1}}\be$ is constant, 
implies the third condition, that $\be(v_i^\top)-\la(x_i)$ is constant for all $i=1,\dots,k$.
This reflects the fact that the move of one of the points on the curve predicts the movement of all the other $k-1$ points.
\end{remark}

%%%%%%%%%%%%%

\paragraph{The symplectic form.}
Our plan is to  show that \eqref{cbx} is a momentum map for a well chosen symplectic form on the orbit $\O_a^{\bar \w}$.
Let $\ev_{t_i}:\Emb_a(S^1,\RR^2)\to\RR^2$ denote the evaluation map at the points $t_i\in S^1$ given by \eqref{ti},
so the collection of maps $(\ev_{t_i})$
%$(\ev_{t_i})(f)=(f(t_i))\in\Conf_k$, 
takes values in the configuration space $\Conf_k$.
%of $k$ points in the plane.
On $\Emb_a(S^1,\RR^2)$ we consider the 2-form 
\begin{equation}\label{short}
\Om^{\Ga\bar\w}=\Om+(\ev_{t_i})^*\om^\Ga=\Om+\sum_{i=1}^k\Ga_i\ev_{t_i}^*\om,
\end{equation}
where $\Omega$ is the 2-form \eqref{omf}  restricted to embeddings that enclose a fixed area $a$,
and $\om^\Ga$ is the symplectic form \eqref{omsum} on the configuration space $\Conf_k$.
More precisely,
\begin{align}\label{omag}
(\Om^{\Ga\bar\w})_f(u_f, v_f) 
%&= \Om_f(u_f,v_f) + (\om^\Ga)_{(f(t_i))}(u_f(t_i),v_f(t_i))\nonumber \\
&=\frac{\w}{2\pi}\int_{S^1}\omega(u_f(t), v_f(t)) dt + \sum_{i=1}^k\Gamma_i\omega(u_f(t_i), v_f(t_i)).
\end{align}
Given the fact that $\omega=d\nu$ is an exact form, $\Om^{\Ga\bar\w}$ is also exact, hence closed
(the hat calculus from \cite{V} can be useful here). 
%More precisely, $\Om^{\Ga\bar\w}=d(\Th+\sum_{i=1}^k\Ga_i(\ev_{t_i})^*\nu)$ for the 1-form
%\[\Th_f(u_f)=\frac{\w}{2\pi}\int_{S^1}\nu(u_f(t)) dt .\]

To prove the non-degeneracy of $\Om^{\Ga\bar\w}$, we will use the Euclidean orthogonal decomposition of vector fields along the curve $C=f(S^1)$ into their  normal and tangent parts, together with the non-degenerate pairing from the Appendix.
Let $\{{\bf t},{\bf n}\}$ be the positively oriented orthonormal frame along the curve, where the tangent vector ${\bf t}$ respects the orientation induced by 
the volume form $\be$ on the curve. In particular $\om({\bf t},{\bf n})=1$.
We decompose $u_f\in T_f\Emb_a(S^1,\RR^2)$ into its normal and tangential components
\begin{equation}\label{uft}
u_f=(\rh {\bf n}+\la {\bf t})\o f.
\end{equation}
Because of the fixed enclosed area,  the normal variation $\rh$ is performed with  zero integral functions in $C_0^\oo(C):=\{\rh\in C^\oo(C):\int_C\rh\mu_C=0\}$,
where $\mu_C=i_{\bf n}\om$ denotes the volume form on $C$ induced by the Euclidean metric,
%With the decomposition \eqref{uft}, the tangent space to the manifold of embeddings that enclose a fixed area is
so the tangent space decomposes as
\begin{equation}\label{ctc}
T_f\Emb_a(S^1,\RR^2) \cong C_0^\infty(C) \times C^\infty(C).
%\quad u_f\leftrightarrow (\rh,\la),
\end{equation}
The symplectic form  \eqref{omag} becomes
\[
(\Om^{\Ga\bar\w})_f((\rho_1,\lambda_1),(\rho_2,\lambda_2))
=\frac{\w}{2\pi}\int_C(\rh_2\la_1-\rh_1\la_2)\be
+\sum_{i=1}^k\Ga_i (\rh_2\la_1-\rh_1\la_2)(x_i),
\]
where $x_i=f(t_i)$ are $k$ points on the curve $C=f(S^1)$.
We define the pairing 
\begin{equation}\label{par}
\langle\langle \ ,\  \rangle\rangle: C_0^\infty(C) \times C^\infty(C) \rightarrow \RR, \quad
\langle\langle \rho,\lambda \rangle\rangle := \frac{\w}{2\pi}\int_{C}\rh\la\be+
 \sum_{i=1}^k\Gamma_i(\rho\lambda)(x_i),
\end{equation}
which is non-degenerate by the Lemma \ref{app}.
It allows us to rewrite the 2-form  $\Om^{\Ga\bar\w}_a$ in \eqref{omag} at $f$ as
\begin{equation}\label{pair}
(\Om^{\Ga\bar\w})_f((\rho_1,\lambda_1),(\rho_2,\lambda_2)) = \langle\langle \rh_2,\la_1 \rangle\rangle - \langle\langle \rh_1,\la_2 \rangle\rangle
\end{equation}
and we conclude that:

\begin{proposition}
The  2-form  $\Om^{\Ga\bar\w}$ on $\Emb_a(S^1,\RR^2)$ given in \eqref{omag} is symplectic.
\end{proposition}

\begin{remark}\label{yael}\rm
Let us denote by the same symbol $\Om^{\Ga\bar\w}$ the symplectic form on $\O_a^{\bar\w}$ transported 
%from $\Emb_a(S^1,\RR^2)$ 
with the bijection $\Ph$, and let us use the decomposition \eqref{ctc} also for the tangent space $T_{(C,\be,(x_i))}\O_a^{\Ga\bar\w}$. The infinitesimal generator $f\mapsto f'$ for the $\Rot(S^1)$ action on $\Emb_a(S^1,\RR^2)$
corresponds to $(0,|f'|)$.
% in the decomposition \eqref{ctc}. 
It moves all the points $x_1,\dots,x_k$ on the curve without changing $(C,\be)$.
If we want to move a single point $x_1$ on the curve, without changing the curve and the other $k-1$ points,
then we have to change $\be$ because $\int_{x_i}^{x_{i+1}}\be$ is constant.
A tangent vector at $(C,\be,(x_i))$ that does this is of the form $(0,\la)$ in the decomposition \eqref{ctc},
 with $\la\in C^\oo(C)$ supported in a neighborhood of $x_1$ that doesn't contain any of the points $x_2,\dots,x_k$.
A symplectically conjugate direction 
%to moving the point $x_1\in C$ on the curve, by letting both the curve and all the remaining points $x_i$ unchanged,
% (this forces all the other points $x_i$ to move accordingly, since the local vorticities are fixed) 
requires to deform the curve around $x_1$.
Indeed, in this case the decomposed vector $(\bar\rh,\bar\la)$ is chosen such that the support of $\bar\rh$ contains a neighborhood of $x_1$, 
and in this way we can achieve that
\[
(\Om^{\Ga\bar\w})_{(C,\be,(x_i))}((0,\la),(\bar\rho,\bar\la))=\langle\langle \bar\rho,\lambda \rangle\rangle 
= \frac{\w}{2\pi}\int_{C}\bar\rho\lambda\be+ \Gamma_1(\bar\rho\lambda)(x_1)\ne 0.
\]
\end{remark}

%\todo{2}
\begin{remark}\rm
Another formulation for the symplectic form associated to pointed vortex loops uses
the decomposition of the tangent space in terms of triples \eqref{trip}.
The symplectic form \eqref{omag} becomes
\begin{align*}
(\Om^{\Ga\bar\w})_{(C,\be,x)}((u_C,d\la,(v_i)),(\bar u_C,d\bar\la,(\bar v_i)))&=\int_C(\bar\la i_{u_C}\om-\la i_{\bar u_C}\om)\\
&+\sum_{i=1}^k\Ga_i(\om(u_C(x_i),\bar v_i)-\om(\bar u_C(x_i),v_i)).
\end{align*}
\end{remark}

%%%%%%

\paragraph{Coadjoint orbit.}
It is easy to see that the natural $G$ action 
$\ph\cdot f=\ph\o f$ 
on the symplectic manifold $(\Emb_a(S^1,\RR^2),\Om^{\Ga\bar\w})$
is Hamiltonian
%\begin{align*}(\Omega^{\Ga\bar\w})_f(X_h\circ f, v_f)&=\frac{\w}{2\pi}\int_{S^1}(d_{f(t)}h)(v_f(t))dt + \sum_{i=1}^k \Ga_i(d_{f(t_i)}h)(v_f(t_i))\end{align*}
with equivariant momentum map 
\begin{equation}\label{mmappvl}
J: \Emb_a(S^1, \RR^2) \rightarrow \g^*, \quad\langle J(f), h \rangle = \frac{\w}{2\pi}\int_{S^1}h(f(t))dt + \sum_{i=1}^k\Gamma_ih(f(t_i)).
\end{equation}
%One can easily check that $J$ is equivariant.
Under the identification $\Ph$ in \eqref{fie} of $\Emb_a(S^1,\RR^2)$ with $\O_a^{\bar\w}$, the momentum map \eqref{mmappvl} is exactly the map \eqref{cbx}.

The injectivity of $J$ depends on the rotational symmetry of the data configurations $(\Ga_i)$ and $(\w_i)$.
Let $\ell\in\{1,\dots,k\}$ be the smallest natural number that satisfies
\begin{equation}\label{sym}
\Ga_{i+\ell}=\Ga_i,\quad \w_{i+\ell}=\w_i, \text{ for all }i\in\{1,\dots,k\},
\end{equation}
identities fulfilled for $\ell=k$ by convention.
%with the convention that $\Ga_{i+k}=\Ga_i$ and $\w_{i+k}=\w_i$ for all $i$.
Then $\ell$ must be a divisor of $k$,
%Since $k$ automatically satisfies \eqref{sym}, 
because the biggest common divisor of $\ell$ and $k$ also satisfies \eqref{sym},
so it must be equal to $\ell$. 

Let $m=k/\ell$. Then $\w_1+\dots+\w_\ell={\w}/{m}$, 
and the same holds for the sum of any $\ell$ consecutive $\w_i$'s, so 
$$t_{i+\ell}=t_i+\frac{2\pi}{\w}(\w_i+\dots+\w_{i+\ell-1})
%=t_i+\frac{2\pi}{\w}(\w_1+\dots+\w_{\ell})
=t_i+\frac{2\pi}{m}.$$
It follows that the subgroup 
of $\Rot(S^1)$ generated by the rotation  $R_{2\pi/m}$, which is isomorphic to the cyclic group $\ZZ_m$,
preserves the symplectic form $\Om^{\Ga\bar\w}$ in \eqref{omag}. 
Thus the $G$ action on  the quotient manifold $\Emb_a(S^1,\RR^2)/\ZZ_m$ is Hamiltonian with momentum map descending from \eqref{mmappvl}:
\begin{equation}\label{zem}
\bar J: \Emb_a(S^1, \RR^2)/\ZZ_m \rightarrow \g^*, \quad\langle\bar J([f]), h \rangle = \frac{\w}{2\pi}\int_{S^1}h(f(t))dt + \sum_{i=1}^k\Gamma_ih(f(t_i)).
\end{equation}
The two extreme cases are:
(i)  the generic case, for data configurations $(\Ga_i)$ and $(\w_i)$ without rotational symmetry, with $\ell=k$ and  $m=1$,
(ii) the exceptional case, for equal point vorticities $\Ga_i$ and equal partial vorticities $\w_i$, with $\ell=1$ and $m=k$. 

\begin{lemma}\label{lemi}
The equivariant momentum map $\bar J$ is injective. 
\end{lemma}

\begin{proof}
Let $f_1,f_2\in\Emb_a(S^1,\RR^2)$ with $\bar J([f_1])=\bar J([f_2])$, which means  that $J(f_1)=J(f_2)$. 
Thus for all $h\in C^\oo_c(\RR^2)$,
\begin{equation}\label{jef}
\frac{\w}{2\pi}\int_0^{2\pi}h(f_1(t)) dt + \sum_{i=1}^k \Gamma_ih(f_1(t_i)) = \frac{\w}{2\pi}\int_0^{2\pi}h(f_2(t)) dt + \sum_{i=1}^k\Gamma_ih(f_2(t_i)),
\end{equation}
with $t_i=\tfrac{2\pi}{\w}(\w_1+\dots+\w_{i-1})$.
Assume by contradiction that the curves $C_1=f_1(S^1)$ and $C_2=f_2(S^1)$ do not coincide.
%$\text{Im}(f_1) \neq \text{Im}(f_2)$ and denote the curves these two images describe by $C_1$, respectively by $C_2$. Since they don't coincide, 
There exists a point $x$ contained in the second, but not in the first, which means that exists a whole neighborhood of $x$ contained in $C_2$ but not in $C_1$.
Shrinking this neighborhood, we get a subset $U\subset C_2 \setminus C_1$ such that $f_2(t_i)\notin U$ for all $i$.
%In particular also $f_1(t_i)\notin  U$ for all $i$. 
Now let $h \in C_c^\infty(\RR^2)$ be a non-negative function supported in $U$. In particular, $h$ vanishes on $C_1$.
Under these circumstances 
\begin{align*}
J(f_1) = J(f_2) &\Rightarrow 
\frac{\w}{2\pi}\int_0^{2\pi}h(f_2(t)) dt = 0,
\end{align*}
which contradicts the choice of $h$. The assumption has been false, so the embeddings $f_1,f_2$ have the same image.

The existence of a diffeomorphism $\gamma\in \Diff(S^1)$ such that $f_2 = f_1 \circ \gamma$ follows. We show that $\ga'=1$, \ie $\ga$ is a rotation.
We assume by contradiction that there exists $t_0\in S^1$ such that $\gamma'(t_0)\neq 1$. Then there is a whole neighborhood of $t_0$ in $S^1$ on which 
$\gamma^\prime-1$ doesn't change sign.
Shrinking it, we get a subset $V$ of $S^1$ such that  $t_i\notin V$ and $\gamma(t_i) \notin V$ for all $i$.
Any non-negative function $h \in C_c^\infty(\RR^2)$ such that $h\o f_2$ is supported in $V$ satisfies
\begin{equation*}
\int_0^{2\pi}(h\o f_1)(t)) dt  =\int_0^{2\pi}(h\o f_2)(\gamma^{-1}(t)) dt = \int_0^{2\pi}(h\o f_1)(t){\gamma}^\prime(t) dt \ne \int_0^{2\pi}(h\o f_2)(t)dt.
\end{equation*}
This contradicts the identity $\int_0^{2\pi}(h\o f_1)(t)) dt  = \int_0^{2\pi}(h\o f_2)(t)dt$ that follows from \eqref{jef}.
Thus $$\gamma^\prime = 1 \Rightarrow \gamma\in \text{Rot}(S^1) \Rightarrow f_2 = f_1 \circ R_\tau.$$

It remains to prove that $\ta$ is a multiple of $2\pi/m$.
With $f_2 = f_1 \circ R_\tau$, the identity \eqref{jef} becomes
\[
\sum_{i=1}^k \Gamma_i\de_{f_1(t_i) }= \sum_{i=1}^k\Gamma_i\de_{f_1(t_i+\ta)}. 
\]
Since $f_1$ is an embedding, the rotation $R_\ta$ has to make a  rotational permutation of the points $t_i$ on the circle,
while preserving the corresponding vorticities. More precisely, there exists a natural number $\ell(\ta)$ such that for all $i$:
\[
t_{i+\ell(\ta)}=t_i+\ta,\quad \Ga_{i+\ell(\ta)}=\Ga_i.
\]
It follows that  $t_{i+\ell(\ta)+1}-t_{i+\ell(\ta)}=t_{i+1}-t_i$, so $\w_{i+\ell(\ta)}=\w_i$ for all $i$.
Now the symmetry conditions \eqref{sym} are fulfilled for $\ell(\ta)$,
which implies that $\ell(\ta)$ is a multiple of $\ell$.
Writing $\ell(\ta)=j\ell$ for $j\in\NN$, we get
\[
\ta=t_{\ell(\ta)+1}-t_1=\frac{2\pi}{\w}(\w_1+\dots+\w_{\ell(\ta)})\stackrel{\eqref{sym}}=\frac{2\pi}{\w}j(\w_1+\dots+\w_{\ell})
=\frac{2\pi}{m}j,
\]
as desired. Hence $R_{\ta}=(R_{{2\pi}/{m}})^j$ and $[f_1]=[f_2]$.
This ensures the injectivity of $\bar J$.
\end{proof}

\begin{theorem}\label{main}
The space $\O_a^{\bar\w}$ of pointed vortex loops with enclosed area $a$ and partial weights $\w_i$, $i=1,\dots,k$,
identified with $\Emb_a(S^1,\RR^2)$ 
and endowed with the symplectic form $\Om^{\Ga\bar\w}$ in \eqref{omag},
has a natural Hamiltonian action of $\Ham_c(\RR^2)$. We say that the data $(\Ga_i),(\w_i)$ have rotational symmetry if there exists a natural number $0<\ell<k$ such that $\Ga_{i+\ell}=\Ga_i$ and $\w_{i+\ell}=\w_i$ for all $i$, and we choose the smallest $\ell$ with these properties.
\begin{enumerate}
\item
If the data $(\Ga_i),(\w_i)$ have no rotational symmetry, then $\Emb_a(S^1,\RR^2)$ can be realized as a coadjoint orbit.
More precisely, the momentum map $J$ in \eqref{mmappvl}
is one-to-one onto a coadjoint orbit and the KKS symplectic form satisfies 
$J^*\om_{KKS} = \Om^{\Ga\bar\w}$. 
\item
If the data $(\Ga_i),(\w_i)$ have rotational symmetry, then $k=m\ell$ for a natural number $m>1$,
and the quotient space $\Emb_a(S^1,\RR^2)/\ZZ_m$, with $\ZZ_m$ the rotation subgroup  generated by the rotation $R_{2\pi/m}$, 
%(isomorphic to $\ZZ_m$) 
can be realized as a coadjoint orbit. More precisely,
the momentum map $\bar J$ in \eqref{zem}
is one-to-one onto a coadjoint orbit and the KKS symplectic form satisfies 
$\bar J^*\om_{KKS} = \Om^{\Ga\bar\w}$. 
\end{enumerate}
\end{theorem}

This follows from Lemma \ref{lemi} together with  the following mathematical folklore result (see for instance the Appendix in \cite{HV}):
\begin{proposition}\label{stefan}
Suppose the action of $G$ on $(\mathcal M,\Om)$ is transitive 
%and infinitesimally transitive, 
with injective equivariant moment map $J:\mathcal M\to\g^*$.
Then $J$ is one-to-one onto a coadjoint orbit of $G$.
Moreover, it pulls back the Kostant--Kirillov--Souriau symplectic form $\omega_{\KKS}$ on the coadjoint orbit to the symplectic form $\Om$.
\end{proposition}

\begin{remark}\rm
In the first case of the Theorem \ref{main}, the same differential manifold $\Emb_a(S^1,\RR^2)$ serves as a model for several coadjoint orbits.
They have non-equivalent KKS symplectic forms $\Om^{\Ga\bar\w}=\Om+\sum_{i=1}^k\Ga_i\ev_{t_i}^*\om$, 
unless all the $\Ga_i$'s and $\w_i$'s coincide.
This is clear for the point vorticities $\Ga_i$. An argument for the partial vorticities $\w_i$ goes as follows: 
a different set $w'_i$ defines a different set of coordinates $t'_i$ on the circle by \eqref{ti}, starting with $t'_1=t_1=0$.
Let $\ga\in\Diff(S^1)$ be a diffeomorphism that maps each $t_i$ to its corresponding $t'_i$.
The induced diffeomorphism of $\Emb_a(S^1,\RR^2)$ given by reparametrization, $f\mapsto f\o\ga$,  pulls back 
$\sum_{i=1}^k\Ga_i\ev_{t'_i}^*\om$ to $\sum_{i=1}^k\Ga_i\ev_{t_i}^*\om$.
But it doesn't preserve $\Om$, unless $\ga$ is a rigid rotation (see \cite{V}). This is impossible because $\ga(0)=0$. 
As a consequence it doesn't preserve $\Om^{\Ga\bar\w}$ either.
\end{remark}

\begin{remark}\rm
The circle bundle $\pi_{\can}:\Emb_a(S^1,\RR^2)\to\Emb_a(S^1,\RR^2)/\Rot(S^1)$
has an unexpected feature: both the total space and the base manifold can be identified with coadjoint orbits, hence both are symplectic manifolds,
even though the fibre is 1-dimensional.
% $\O_a^{\w}$, but also the total space is a coadjoint orbit $\O_a^{\Ga\bar\w}$
\end{remark}

\begin{remark}\rm
There is no direct link to the quotient space $\Conf_k/\S_K$ from Proposition \ref{permit},
since the confinement of the $k$ points $x_i$ to the curve  $C$ forces the permutation subgroup to respect their ordering.
But of course $\ZZ_m$ can be seen as a subgroup of the permutation group $\S_k$ acting on the configuration space $\Conf_k$.
It is generated by the following product of $\ell$ cycles of length $m$:
\[
(1,\ell+1,\dots,(m-1)\ell+1)(2,\ell+2,\dots,(m-1)\ell+2),\dots (\ell,2\ell,\dots,k).
\]
\end{remark}
%%%

\paragraph{Symplectic embedding.}
In the previous section we have described the symplectic manifolds $(\Conf_k, \om^{\Gamma})$ and $(\O_a^{\w}, \Omega_{a}^{\w})$ 
%are symplectic manifolds, describing the coadjoint orbits of
that consist of point vortices and vortex loops. Their Cartesian product is a symplectic manifold in a canonical way:
\begin{equation}\label{product}
(\O_a^{\w}\x\Conf_k, \ \Om_{\prod}:=\pr_1^*\Omega_a^{\w} + \pr_2^*\om^\Ga),
\end{equation}
where $\text{pr}_1$ and $\text{pr}_2$ are the projections on the first, respectively on the second factor. 

%Let us take vorticity data $(\Ga_i),(\w_i)$ without rotational symmetries.
We consider the injective map that splits a pointed vortex loop $(C,\be,(x_i))$ into its loop part $(C,\be)$ and its additional vortex points $(x_1,\dots,x_k)$:
\[
j:\O_a^{\Ga\bar\w}\to\O_a^{\w_1+\dots+\w_k}\x\Conf_k,\quad j(C,\be,(x_i))=((C,\be),(x_i)).
\]
With the bijections $\Ph$ and $\Ps$ from  \eqref{fie} and \eqref{psi}, we write it as:
%the following map, denoted by the same letter:
\[
j:\Emb_a(S^1,\RR^2)\to(\Emb_a(S^1,\RR^2)/\Rot(S^1))\x\Conf_k,\quad j(f)=([f],(f(t_i))).
\]

\begin{proposition}\label{sub}
The symplectic manifold $(\O_a^{\bar\w}, \Om^{\Ga\bar\w})$ of pointed vortex loops is a symplectic submanifold of 
the product of the symplectic manifolds $(\O_a^{\w_1+\dots+\w_k},\Om_a^{\w_1+\dots+\w_k})$ of vortex loops 
and $(\Conf_k,\om^\Ga)$ of vortex points.
\end{proposition}

\begin{proof}
We verify that the injective map $j$ is  symplectic.
%We denote by $\pi_{\can}:\Emb_a(S^1,\RR^2)\to\Emb_a(S^1,\RR^2)/\Rot(S^1)$ the canonical projection $f\mapsto[f]$ and by $\ev_t:\Emb(S^1,\RR^2)\to\RR^2$ the evaluation at $t\in S^1$. Thus
We notice that $\pr_1\o j=\pi_{\can}$ the canonical projection, and $\pr_2\o j=(\ev_{t_i})$ a collection of evaluation maps.
Now the computation
\[
j^*\Om_{\prod}
%=(\pr_1\o j)^*\Om_a^{\w}+(\pr_2\o j)^*\om^\Ga
=\pi_{\can}^*\Om_a^{\w}+(\ev_{t_i})^*\om^\Ga=\Om+\sum_{i=1}^k\Ga_i\ev_{t_i}^*\om
=\Om^{\Ga\bar\w}
\]
ensures that $j$ is a symplectic map.
\end{proof}

\begin{remark}\rm
Even though the symplectic form $\Om_{\prod}$ is not exact (because $\Om_a^{\w}$ is not exact by \cite{GBVpre}),
still the pullback $j^*\Om_{\prod}=\Om^{\Ga\bar\w}$ is exact.
\end{remark}

Since $j$ is $G$ equivariant and symplectic, the fact that the momentum map \eqref{mmappvl} is built with the momentum maps \eqref{jloop} and \eqref{jpoint} is by no means a  surprising fact.

%%%%%

\paragraph{Polarization.}
In \cite{goldin} is shown that, within the framework of the 2D Euler equations, point vortices cannot be consistently quantized, 
while vortex loops admit natural polarizations. They notice that group $G_C$ of Hamiltonian diffeomorphisms that preserve the curve $C$ is a polarization subgroup of $G$ for the coadjoint orbit $\O_a^{\w}$ of the vortex loop $(C,\be)\in\g^*$, as noticed  in \cite{goldin}. 
The resulting configuration space is the space of loops enclosing a fixed area, without information about the vorticity distribution.

Here we show that pointed vortex loops also have natural polarizations.
More precisely, the same polarization subgroup $G_C$, that consists of diffeomorphisms that preserve the loop as a set,
works for the coadjoint orbit $\O_a^{\bar\w}$ of the pointed vortex loop 
$(C,\be,(x_i))$ from \eqref{cbx}.
Thus the configuration space is the space of loops that enclose a fixed area,
without information about the vorticity distribution and the attached points.

\begin{proposition}\label{max}
The group $G_C$ of Hamiltonian diffeomorphisms that preserve the curve $C$ is a polarization subgroup  for the coadjoint orbit
of the pointed vortex loop $(C,\be,(x_i))\in\g^*$.
\end{proposition}

\begin{proof}
The Lie algebra $\g_C$ of $G_C$ consists of compactly supported Hamiltonian vector fields that are tangent to the curve $C$.
It satisfies the polarization condition:
\[
\langle (C,\be,(x_i)),[X_{h_1},X_{h_2}]\rangle=\int_C\om(X_{h_1},X_{h_2})\be
+\sum_{i=1}^k\Ga_i\om(X_{h_1},X_{h_2})(x_i)
=0,
\]
for all $X_{h_1},X_{h_2}\in\g_C$,
because along $C$ the two Hamiltonian vector fields are linear dependent and $x_i\in C$.
We also use the fact that $\om(X_{h_1},X_{h_2})$ is a Hamiltonian function for $[X_{h_1},X_{h_2}]$.

Let us assume by contradiction that there exists a Lie subalgebra $\h\subset\g$ that satisfies the polarization condition
and is strictly bigger than $\g_C$. This means we find $X_{h_0}\in\h$ and $x_0\in C$ such that $X_{h_0}(x_0)\notin T_{x_0}C$ (so $X_{h_0}$ is not tangent to $C$ in a whole neighborhood of $x_0$, which we shrink so that it doesn't contain any of the points $x_1,\dots,x_k$).
But there exists $X_h\in\g_C$ with support in this neighborhood. Then $\langle (C,\be,(x_i)),[X_{h_0},X_{h}]\rangle=\int_C\om(X_{h_0},X_{h})\be\ne 0$,
which contradicts the polarization condition for $\h$. This shows the maximality of $\g_C$.
\end{proof}

Thus the configuration manifold for both $\O_{a}^{\w}$ and $\O_a^{\bar\w}$ is $G/G_C$, identified with the space of curves with fixed enclosed area $a$.   
These results fit well with the canonical expressions of the corresponding KKS symplectic forms:
\eqref{omaw} on the tangent space $C_0^\oo(C)\x dC^\oo(C)$ to $\O_{a}^{\w}$,
%for  $\Om_a^{\w}$ 
and \eqref{pair} on the tangent space $C_0^\oo(C)\x C^\oo(C)$ to $\O_a^{\bar\w}$,
%for $\Om^{\Ga\bar\w}$,
where $C^\oo_0(C)$ can be identified with the tangent space at the point $C$ to the configuration manifold $G/G_C$.

\appendix
\section{Non-degenerate pairings}

%\todo{de eliminat partea cu vortex loops?}
\paragraph{Non-degenerate pairing for vortex loops.}
Under the decomposition in the orthogonal base $\{f',{\bf n}\}$ along the loop parametrized by $f$, 
the tangent space to the manifold of embeddings can be identified with the Cartesian product
\begin{equation}\label{ctc2}
T_f\Emb(S^1,\RR^2) \cong C^\infty(C) \times C^\infty(C),\quad u_f=(\rh\o f)({\bf n}\o f)+(\la\o f) f'.
\end{equation}
We can write now the symplectic form \eqref{omf} at $f$ in the canonical form
$\Om_f((\rh_1,\la_1),(\rh_2,\la_2))=\langle \rho_2,\lambda_1 \rangle - \langle \rho_1,\lambda_2 \rangle$,
for the $L^2$ scalar product with respect to the  volume form $\mu_C$ induced on the curve by the Euclidean metric:
\begin{equation}\label{l2}
\langle\rh,\la\rangle=\frac{\w}{2\pi}\int_C\la\rh\mu_C.
\end{equation}
The identity $f^*\mu_C=\om({\bf n}\o f,f')dt$ is used here.

The infinitesimal generator for the $\Rot(S^1)$ action $f\mapsto f'$ corresponds to the pair $(0,1)$ under the decomposition \eqref{ctc2} at each $f$.
Thus we further identify
\[
T_{[f]}(\Emb_a(S^1,\RR^2)/\Rot(S^1))\cong C_0^\oo(C)\x (C^\oo(C)/\RR)\cong C_0^\oo(C)\x dC^\oo(C),
\]
where $C_0^\oo(C)$ denotes the subspace of zero integral functions with respect to the induced volume form $\mu_C$.
The reduced symplectic form $\Om_a^{\w}$ at $[f]$ takes the canonical form  (as in \cite{GBV})
\begin{equation}\label{omaw}
(\Om_a^{\w})_{[f]}((\rh_1,d\la_1),(\rh_2,d\la_2))=\langle \rho_2,d\lambda_1 \rangle - \langle \rho_1,d\lambda_2 \rangle,
\end{equation}
%he $L^2$ scalar product \eqref{l2} induces a
for the non-degenerate pairing induced by \eqref{l2} between zero-integral functions and exact 1-forms 
$\langle\rh,d\la\rangle=\langle\rh,\la\rangle$.

\paragraph{Non-degenerate pairing for pointed vortex loops.}
We define another non-degenerate pairing of $C_0^\oo(C)$, this time with $C^\oo(C)$, which brings the symplectic form $\Om^{\Ga\bar\w}$
from \eqref{omag} 
at $f\in \Emb_a(S^1,\RR^2)$ in the canonical form \eqref{pair}.

Let $t_1, \dots,t_k$ be consecutive points on the circle $S^1=\RR/2\pi\ZZ$ (\ie $0\le t_1\le\dots\le t_k<2\pi$).
Let $c,c_1,\dots,c_k$ be non-zero real numbers.
% and let $c>0$.

\begin{lemma}\label{app}
The pairing $\langle\langle \ ,\  \rangle\rangle$ between the space of smooth functions on $S^1$ and its subspace of zero integral functions 
$C_0^\oo(S^1)=\{\rh\in C^\oo(S^1):\int_{S^1}\rh(t)dt=0\}$,
\[
\langle\langle \ ,\  \rangle\rangle: C_0^\infty(S^1) \times C^\infty(S^1) \rightarrow \RR, \quad
\langle\langle \rho,\lambda \rangle\rangle= 
c\int_{S^1}\rh(t)\la(t)dt + \sum_{i=1}^kc_i\rh(t_i)\la(t_i),
\]
is non-degenerate. 
\end{lemma}

\begin{proof}
Let $\rho_0 \in C_0^\infty(S^1)$ such that $\langle\langle \rho_0, \lambda \rangle\rangle = 0$, $\forall \lambda \in C^\infty(S^1)$,
which means that
\begin{equation}\label{rzpair}
c\int_{S^1}\rh_0(t)\la(t)dt + \sum_{i=1}^kc_i\rh_0(t_i)\la(t_i) = 0.
\end{equation}
%n particular any function  $\lambda$ in $C^\infty(S^1)$ such that $\lambda(t_i) = 0$, $\forall i$, satisfies
%\begin{equation}\label{rhozero}\int_{S^1}\rho_0(t)\lambda(t) dt = 0.\end{equation}
Assume by contradiction that there exists $t \notin \{t_1, ..., t_k\}$ such that $\rho_0(t)\neq 0$. 
Choosing a non-negative bump function $\lambda \in C^\infty(S^1)$, supported in a neighborhood of $t$ that doesn't contain any  of the $t_i$'s 
and where $\rho_0$ doesn't change sign, we get $\int_{S_1}\rho_0(t)\lambda(t) dt \neq 0$,
which contradicts equation $\eqref{rzpair}$. 
%The assumption was false. 
We conclude that $\rho_0$ vanishes on $S^1\setminus\{t_1,...,t_k\}$ and, because $\rho_0$ is a continuous function, this implies $\rho_0 = 0$. 

Now let $\lambda_0 \in C^\infty(S^1)$ such that $\langle\langle \rho, \lambda_0 \rangle\rangle  = 0$, $\forall \rho \in C_0^\infty(S^1)$,
which means that
\begin{equation}\label{lzpair}
c\int_{S^1}\rh(t)\la_0(t)dt + \sum_{i=1}^kc_i\rh(t_i)\la_0(t_i) = 0.
\end{equation}
%Let $\rho$ be an arbitrary function in $C_0^\infty(S^1)$ with $\rho(t_i) = 0$, $\forall i$. Then: 
%\begin{equation}\label{lamzero}\int_{S^1}\rho(t)\lambda_0(t) dt = 0.\end{equation}
Assume by contradiction that $\lambda_0$ is not constant on $S^1\setminus \{t_1,...,t_k\}$, 
hence there exist $t^\prime \neq t^{\prime\prime}$ distinct from $t_1,\dots,t_k$, such that $\lambda_0(t^\prime) \neq \lambda_0(t^{\prime\prime})$. 
Choose a "double bump" zero integral function $\rho \in C_0^\infty(S^1)$ supported in the union of small neighborhoods of $t^\prime$ and $t^{\prime\prime}$ that don't contain any of the $t_i$'s. 
By shrinking the support of $\rho$ around the points $t'$ and $t''$, we can achieve 
\[
\int_{S^1}\rho(t)\lambda_0(t) dt \rightarrow \lambda_0(t^\prime) - \lambda_0(t^{\prime\prime}) \neq 0,
\]
which contradicts relation $\eqref{lzpair}$. 
%The assumption was false. 
We conclude that $\lambda_0$ is constant on $S^1\setminus \{t_1,...,t_k\}$. Since $\lambda_0$ is a continuous function, this means $\lambda_0$ is constant on $S^1$. 

Going back to equation $\eqref{lzpair}$, we get
$\la_0\sum_{i=1}^kc_i\rho(t_i) = 0$ for all zero integral functions $\rh$.
Pick $\rho\in C_0^\infty(S^1)$ with $\rho(t_1) \neq 0$ and $\rho(t_i) = 0$ for all the other $t_i$. 
From $\la_0c_1\rho(t_1) = 0 $ and $c_1\ne 0$ we conclude that $\la_0=0$.
\end{proof}

\section{A transitivity result}

We prove transitivity of the $\Ham_c(\RR^2)$ action on $\Emb_a(S^1,\RR^2)$
by applying the following lemma (see the appendix in \cite{HVpre}):

\begin{lemma}\label{smoothsec}
Let $G$ be a regular Lie group acting on a smooth manifold $M$ with infinitesimal action $\ze:\g\to\X(M)$. 
Suppose that every point $x_0$ in $M$ admits an open neighborhood $U$ and a smooth map $\sigma:TM|_U\rightarrow \frak g$ such that
\begin{equation}\label{zs}
\zeta_{\sigma(X)}(x)=X
\end{equation}
for all $x\in U$ and $X\in T_xM$. Then the $G$ action on $M$ admits local smooth sections.
\end{lemma}

\begin{proposition} \label{transitivity}
The action of the group $\Ham_c(\RR^2)$ on $\Emb_a(S^1, \RR^2)$, is transitive. 
\end{proposition}

\begin{proof} 
%Using lemma $\eqref{smoothsec}$, we will prove a stronger property, i.e the action admits local smooth sections. 
Firstly, we prove that the action is inifinitesimally transitive, which means proving that given any $u_f$ in 
\[
T_f\Emb_a(S^1, \RR^2) = \left\{u_f: S^1 \rightarrow T\RR^2 : \int_{S^1}f^*i_{u_f}\omega = 0 \right\},
\]
there is $h \in C_c^\infty(\RR^2)$, such that $u_f = \zeta_{X_h}(f) = X_h \circ f$,.
% $X_h\in \frak X_{c,\text{ham}}(\RR^2)$. 
Let $\la\in C^\oo(S^1)$ with the property  $f^*i_{{u_f}}\omega = d\la$. 
We extend it to a function $h_1 \in C_c^\infty(\RR^2)$, which means  that $\la = h_1 \circ f$. 

We define a 1-form along the curve $f(S^1)\subset \RR^2$ by
\begin{equation}\label{gam}
\gamma = dh_1\circ f - i_{u_f}\omega \in \Gamma(f^*T^*M).
\end{equation}
It vanishes on vectors tangent to the curve, hence it can be seen as a function on the normal bundle $Tf(S^1)^\bot$, with differential along the zero section $f(S^1)$ equal to the 1-form itself. Thus, on a tubular neighborhood of $f(S^1)$ in $\RR^2$, cutting $\ga$ with a suitable bump function, 
we get $h_2\in C_c^\infty(\RR^2)$ such that $\gamma = dh_2\circ f$. It follows by \eqref{gam} that 
$dh_2\circ f = dh_1\circ f - i_{u_f}\omega$. This implies $i_{u_f}\omega = i_{{X_h}\circ f}\omega$
for $h = h_1 - h_2$. 
We conclude that $u_f = X_h\circ f = \zeta_{X_h}(f)$ is the infinitesimal generator at $f$ for the Hamiltonian vector field $X_h$. 

Using tubular neighborhoods constructed with the help of a Riemannian metric, we see that the function $h$ above 
may be chosen to depend smoothly on $f$ and $u_f$, for $f$ in a sufficiently small open neighborhood of a fixed embedding $f_0\in \Emb_a(S^1, \RR^2)$. 
Thus $\si(u_f)=X_h$ defines a smooth map that satisfies \eqref{zs}. 
With Lemma $\eqref{smoothsec}$ we obtain that the action of $\text{Ham}_c(\RR^2)$ on $\Emb_a(S^1, \RR^2)$ admits local smooth sections,
hence it is transitive. 
\end{proof}

{
\footnotesize

\bibliographystyle{new}

}
\end{document}